\documentclass[11pt,reqno]{amsart} %
\pdfoutput=1
\usepackage{amsmath,amscd}
\usepackage{amsthm}
\usepackage{amssymb}
\usepackage{amsfonts}
\usepackage{latexsym}
\usepackage{url}
\usepackage{latexsym}
\usepackage{color}
\usepackage{enumerate,fancybox}

\usepackage{epsf}
\usepackage{graphicx}

\DeclareMathOperator\dist{dist}
\DeclareMathOperator\diam{diam}

\newcommand{\capD}{\mathrm{cap}(\overline{\mathbb{D}})}

\DeclareMathOperator{\supp}{\mathrm{supp}}
\DeclareMathOperator{\capy}{\mathrm{cap}}

\def\C{{\mathbb C}}

\newtheorem{theorem}{Theorem}[section]
\newtheorem{lemma}[theorem]{Lemma}
\newtheorem{proposition}[theorem]{Proposition}

\newtheorem{example}[theorem]{Example}

\newtheorem{definition}[theorem]{Definition}

\numberwithin{equation}{section}

\def\restr#1{\,\vrule\,\lower1.75ex\hbox{$#1$}}

\def\be{\begin{equation}}
\def\ee{\end{equation}}
\def\bea{\begin{eqnarray}}
\def\eea{\end{eqnarray}}
\def\bean{\begin{eqnarray*}}
\def\eean{\end{eqnarray*}}

\def\ov{\overline}

\font\tenopen = cmbx10
\font\sevenopen = cmbx7
\font\fiveopen = cmbx5
\newfam\openfam

\textfont\openfam = \tenopen
\scriptfont\openfam = \sevenopen
\scriptscriptfont\openfam = \fiveopen

\input colordvi

\begin{document}
\title[]
{Relative Asymptotics of Orthogonal Polynomials for Perturbed Measures}

\date{\today}

\thanks{{\it Acknowledgements.}
The first author was partially supported by the
U.S.\ National Science Foundation grants  DMS-1412428 and DMS-1516400.
The second author was supported by the University of Cyprus grant 3/311-21027.
}

\author[E.B.\ Saff]{E.B.\ Saff}
\address{Center for Constructive Approximation, Department of Mathematics\\ Vanderbilt University\\
         1326 Stevenson Center\\
         37240 Nashville, TN\\
         USA}
         \email{edward.b.saff@Vanderbilt.edu}
\urladdr{http://my.vanderbilt.edu/edsaff/}

\author[N. Stylianopoulos]{N. Stylianopoulos}
\address{Department of Mathematics and Statistics,
         University of Cyprus, P.O. Box 20537, 1678 Nicosia, Cyprus}
\email{nikos@ucy.ac.cy}
\urladdr{http://ucy.ac.cy/\textasciitilde nikos}



\begin{abstract}
We survey and present some new results that are related to the behavior of orthogonal
polynomials in the plane under small perturbations of the measure of orthogonality.
More precisely, we introduce the notion of a
polynomially small (PS) perturbation of a measure. Namely, if $\mu_0 \ge \mu_1$ and
$\{p_n(\mu_j,z)\}_{n=0}^\infty,  j=0,1,$ are the associated orthonormal polynomial
sequences, then $\mu_0$ a PS perturbation of $\mu_1$ if
$\|p_n(\mu_1,\cdot)\|_{L_2(\mu_0-\mu_1)}\to 0$, as $n\to\infty$.
In such a case we establish relative asymptotic results for the two sequences of
orthonormal polynomials. We also provide results dealing with the behaviour of the zeros
of PS perturbations of area orthogonal (Bergman) polynomials.
\end{abstract}

\maketitle
\allowdisplaybreaks

\section{Introduction}\label{sec:intro}
Let $\mu_0$ and $\mu_1$ be two finite Borel measures having compact and infinite supports
$S_j:=\supp(\mu_j)$ in the complex plane $\mathbb{C}$, with
$\mu_0 \geq \mu_1.$  Then there exists a measure $\mu_2$ such that

\begin{equation}\label{eq:mu-sum}
\mu_0:=\mu_1+\mu_2,
\end{equation}
and we denote the support of $\mu_2$ by $S_2:=\supp(\mu_2)$. We shall regard  $\mu_0$ as a perturbation of $\mu_1$
and investigate when such a perturbation is ``small" in the sense of Definition~\ref{PSP} below.

The three measures yield three Lebesgue spaces $L^2(\mu_j), j=0,1,2,$ with respective inner products
$$
\langle f,g\rangle_{\mu_j} := \int f(z) \overline{g(z)} d\mu_j (z)
$$
and norms
$$
\|f\|_{L^2(\mu_j)}:=\sqrt{\langle f,f\rangle_{\mu_j}}.
$$

Let $\{p_n(\mu_j,z)\}_{n=0}^\infty,  j=0,1,$ denote the sequence of
orthonormal polynomials associated with $\mu_j$; that is, the unique sequence of the form
\begin{equation}\label{def:pnmuj}
p_n(\mu_j,z) = \gamma_n(\mu_j) z^n+ \cdots, \quad \gamma_n(\mu_j)>0,\quad n=0,1,2,\ldots,
\end{equation}
satisfying $\langle p_m(\mu_j,\cdot),p_n(\mu_j,\cdot)\rangle_{\mu_j}=\delta_{m,n}.$\

The corresponding monic polynomials
$p_n(\mu_j,z)/\gamma_n(\mu_j)$, can be equivalently defined by the extremal property
\begin{equation}\label{eq:minimal1}
\left\|\frac{1}{\gamma_n(\mu_j)}{p_n(\mu_j,\cdot)}\right\|_{L^2(\mu_j)}:=
\min_{z^n+\cdots}\|z^n+\cdots\|_{L^2(\mu_j)}.
\end{equation}\

A related extremal problem leads to the sequence $\{\lambda_n(\mu_j,z)\}_{n=0}^\infty$ of the so-called \emph{Christoffel functions}
associated with the measure $\mu_j$. These are defined, for any $z\in\mathbb{C}$, by
\begin{equation}\label{eq:Chrfun-def}
\lambda_n(\mu_j,z):=\inf\{\|P\|_{L^2(\mu_j)}^2,\, P\in\mathbb{P}_n\mbox{ with }  P(z)=1\},
\end{equation}
where $\mathbb{P}_n$ stands for the space of complex polynomials of degree up to $n$.
Using the Cauchy-Schwarz inequality it is easy to verify
(see, e.g., \cite[Section~3]{To05}) that
\begin{equation}\label{eq:Chrfun-pro}
\frac{1}{\lambda_n(\mu_j,z)}=\sum_{k=0}^n|p_k(\mu_j,z)|^2,\quad z\in\mathbb{C}.
\end{equation}
Clearly, $\lambda_n(\mu_j,z)$ is the reciprocal of the diagonal of the kernel polynomial
\begin{equation}\label{eq:kernel-poly-def}
K_n(\mu_j,z,\zeta):=\sum_{k=0}^n \overline{p_k(\mu_j,\zeta)} p_k(\mu_j,z).
\end{equation}
Since $\mu_0\ge\mu_1$, the following inequality is an immediate consequence of (\ref{eq:Chrfun-def})
\begin{equation}\label{eq:comp-pri-lambda}
\lambda_n(\mu_1,z)\le \lambda_n(\mu_0,z),\quad n=0,1,\ldots.
\end{equation}

\begin{definition}\label{PSP}
With the $\mu_j$'s as in \textup{(\ref{eq:mu-sum})},  we say that $\mu_0$ is a
\textbf{polynomially small (PS) perturbation}  of $\mu_1$ provided that $\mu_2$ is not the zero
measure and
\begin{equation}\label{eq:main-gen-asu1}
\lim_{n\to\infty}\|p_n(\mu_1,\cdot)\|_{L^2(\mu_2)}=0.
\end{equation}
\end{definition}

The next result emphasizes the fact that $\mu_0$ being a PS perturbation of $\mu_1$
implies strong constraints on the relative position of the support of $\mu_2$.
Its proof will be given in Section~4.

We  denote by $\Omega$ the unbounded component of $\overline{\mathbb{C}}\setminus S_1$ and
by $\textup{Pc}(S_1)$ the polynomial convex hull of $S_1$, i.e. $\textup{Pc}(S_1):=\overline{\mathbb{C}}\setminus\Omega$.
We use $\capy(E)$ to denote the \emph{(logarithmic) capacity} of a compact set $E$.

\begin{proposition}\label{prop:Co}
If $\mu_0$ is a PS perturbation of $\mu_1$, then
$S_2\subset\textup{Pc}(S_1)$; hence, $S_0\subset\textup{Pc}(S_1)$ and
$\capy(S_0)=\capy(S_1)$.
\end{proposition}

\subsection{Some examples}\label{sec:examples}
Below we provide a list of $\mu_0=\mu_1+\mu_2$, where $\mu_0$ is a PS perturbation
of $\mu_1$.

Throughout the paper we use $A|_E$ to denote the area measure on a bounded set $E$
and $s|_\Gamma$ to denote the arclength measure on a rectifiable curve $\Gamma$.

\begin{itemize}
\item[(i)]
Let $G$ be a bounded Jordan domain (or the union of finitely many bounded Jordan domains with pairwise
disjoint  closures) and let $B$ be a compact subset of $G$. Take $\mu_1=A|_{G\setminus B}$,
$\mu_2=w(z)A|_B$ and $\mu_0=\mu_1+\mu_2$, where  $w(z)$ is integrable on $B$.
Then Lemma~2.2 of \cite{SSST} implies the PS property.
\item[(ii)]
Let $\Gamma$ be a closed piecewise analytic Jordan curve without cusps and let $B$ be a
compact subset in the interior of $\Gamma$. Take $\mu_1=s|_\Gamma$,
$\mu_2=w(z)A|_B$, where $w(z)$ is integrable on $B$. Then, $\mu_0=\mu_1+\mu_2$ is a PS perturbation
of $\mu_1$. (See Theorem~2.1 of \cite{Pr03}.)
\item[(iii)]
Here we assume $\mu_1$ is  in the Szeg\H{o} class on the unit circle; i.e., the absolutely
continuous part
$w(\theta)$ of $\mu_1$ with respect to arclength on the unit circle $|z|=1$ satisfies the condition
$\int_0^{2\pi}\log(w(\theta))d\theta>-\infty$,
and we let $\mu_2$ be a finite measure supported on a compact set inside the unit circle.
Then $\mu_0=\mu_1+\mu_2$ is a PS perturbation of $\mu_1$ thanks to Corollary 2.4.10 of \cite{SimBoI}.
(We remark that \cite{NVY} contains a related result for the case of purely absolutely continuous measures.)
\item[(vi)]
Let $\Gamma$ be a piecewise analytic Jordan curve without cusps, let $G$ denote its interior
and let $\mu_1=A|_G$. If the exterior angle at
$z\in\Gamma$ is less than $\pi/2$, then
$\lim_{n\to\infty}p_n(\mu_1,z)=0$; see \cite[Theorem 1.3]{St16}, and therefore
$\mu_0=\mu_1+t\delta_z$, $t>0$, where $\delta_z$ is the Dirac measure at $z$, is a PS
perturbation of $\mu_1$. On the other hand, if the exterior angle at $z$ is $\pi$, then
$|p_n(\mu_1,z)|\ge Cn^{1/2}$, for some positive constant $C$ and infinitely many $n$;
see \cite[p. 1097]{ToVa15}, and thus $\mu_0$ is \emph{not} a PS perturbation of $\mu_1$,
\end{itemize}

In all the above examples that establish that $\mu_0$ is a PS perturbation of $\mu_1$, the polynomials
$p_n(\mu_1,z)$ go to zero uniformly on the support of
$\mu_2$. The following proposition provides a class of examples of PS perturbations where this
uniform convergence (and in fact pointwise convergence) to zero does not hold uniformly in $S_1$.
Its proof will be given in Section~\ref{Proofs}.

We say that a Jordan curve $\Gamma$  is $C(p,\alpha)$-smooth if $\Gamma$ has an arclength
parametrization $\gamma$ that is $p$-times differentiable and its $p$-th derivative
belongs to the class Lip $\alpha$, $0<\alpha<1$.

\begin{proposition}\label{prop1}
Let $G$ be a bounded Jordan domain with boundary $\Gamma$ in the class $C(2,\alpha)$, $\alpha>1/2$.
If $\mu_1=s|_\Gamma$ is the arclength measure on $\Gamma$ and $\mu_2=A|_G$ is the
area measure on $G$, then
$\mu_0=\mu_1+\mu_2$ is a PS perturbation of $\mu_1$.
\end{proposition}
We remark that under the assumption of the proposition, the inequality
\begin{equation}\label{eq:Suetin-rate}
|p_n(\mu_1,z)|\le c\frac{1}{n^{p+\alpha}},
\end{equation}
for some positive constant $c$, holds locally uniformly inside $\Gamma$; see \cite[Theorem 2.4]{Su66b}.
We also note that for the special case of analytic curves $\Gamma$, Proposition~\ref{prop1} follows
from \cite[Corollary 2.5]{Si13}.

The paper is organised  follows: In Section 2 we state the main results along with some essential lemmas. Section 3
is devoted to the behaviour of the zeros of certain PS perturbations  of Bergman polynomials on the union of Jordan regions.
In Section 4 we provide the proofs of our main results.

\section{Main results on PS perturbations}\label{sec:main}
The main purpose of the paper is to show that the following two associated pairs of sequences
$
\{\gamma_n(\mu_0),\gamma_n(\mu_1)\},\quad
\{p_n(\mu_0,z),p_n(\mu_1,z)\},
$
have comparable asymptotics when the measure $\mu_0$ is a polynomially small perturbation of $\mu_1$.
Furthermore, under a somewhat stronger condition we show that the two Christoffel sequences
$
\{\lambda_n(\mu_0,z),\lambda_n(\mu_1,z)\},
$
likewise have comparable asymptotics. The proofs of our results are given in Section 4.

We use ${\rm Co}(E)$  to denote the convex hull of a set  $E$ and note that $\textup{Co}(S_1)=\textup{Co}(S_0)$;
see Proposition~\ref{prop:Co}.

\begin{theorem}\label{thm:main-gen}
\footnote{This theorem, along with Example~\ref{Ex1}, was presented by the second author
at the CMFT~2013 conference, held in Shantou, China, in June 2013.}
If the measure $\mu_0$ is a PS perturbation of the measure $\mu_1$,
then the following hold:
\begin{itemize}
\item[(i)]
$$
\lim_{n\to\infty}\gamma_n(\mu_1)/\gamma_n(\mu_0)=1;
$$
\item[(ii)]
$$
\lim_{n\to\infty}\|p_n(\mu_1,\cdot)-p_n(\mu_0,\cdot)\|_{L^2(\mu_0)}=0;
$$
\item[(iii)]
uniformly on compact subsets of $\overline{\mathbb{C}}\setminus{\rm Co}(S_0)$ (and for $z\in\Omega$,
provided $p_n(\mu_0,\cdot)$ does not vanish in ${\rm Co}(S_0)\cap\Omega$ for large $n\in\mathbb{N}$),
$$
\lim_{n\to\infty}p_n(\mu_1,z)/p_n(\mu_0,z)=1;
$$
\end{itemize}
Furthermore, if
\begin{equation}\label{eq:main-gen-asu2}
\sum_{j=0}^\infty \|p_j(\mu_1,\cdot)\|_{L^2(\mu_2)}^2<\infty,
\end{equation}
then
\begin{equation}\label{eq:main-gen-asu21}
\lim_{n\to\infty}\lambda_n(\mu_0,z)/\lambda_n(\mu_1,z)=1,
\end{equation}
uniformly on compact subsets of $\Omega$.
\end{theorem}
The appearance of $p_n(\mu_0,z)$ in the denominator of assertion (iii) causes no difficulties, since by
a classical result of Fej\'{e}r \emph{all the zeros of the orthonormal polynomials stay within
the convex hull of the support of the measure of orthogonality} see, e.g., \cite{Sa90}.


Theorem~\ref{thm:main-gen} was motivated by the following example.
\begin{example}\label{Ex1}
Consider the case $G=G_1\cup G_2$, where $G_1$ denotes the canonical pentagon with corners at
the five roots of unity and $G_2$ denotes the disk with center at $3.5$ and radius $2/3$ and let $B$ be
the closed disk inside the pentagon with center at $1/2$ and radius $1/4$.
Then, take $\mu_1=A|_{G\setminus B}=A|_G-A|_B$ and $\mu_2=2A|_B$, so that $\mu_0=A|_G+A|_B$ is a PS perturbation of $\mu_1$, by an application of Lemma~2.2 in \cite{SSST}.
\end{example}

\begin{figure}[t]

\begin{minipage}{\columnwidth}
\begin{center}
\fbox{\includegraphics[scale=0.50]{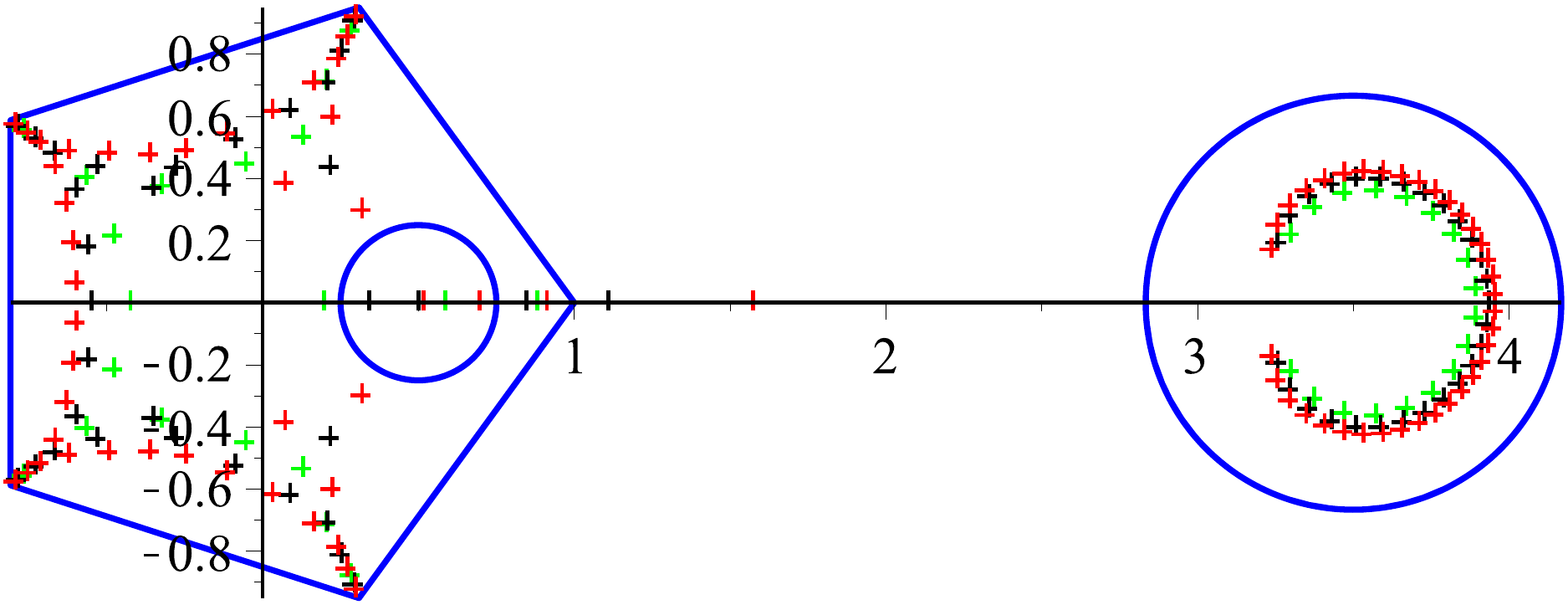}}
\end{center}
\end{minipage}
\begin{minipage}{\columnwidth}
\begin{center}
\fbox{\includegraphics[scale=0.50]{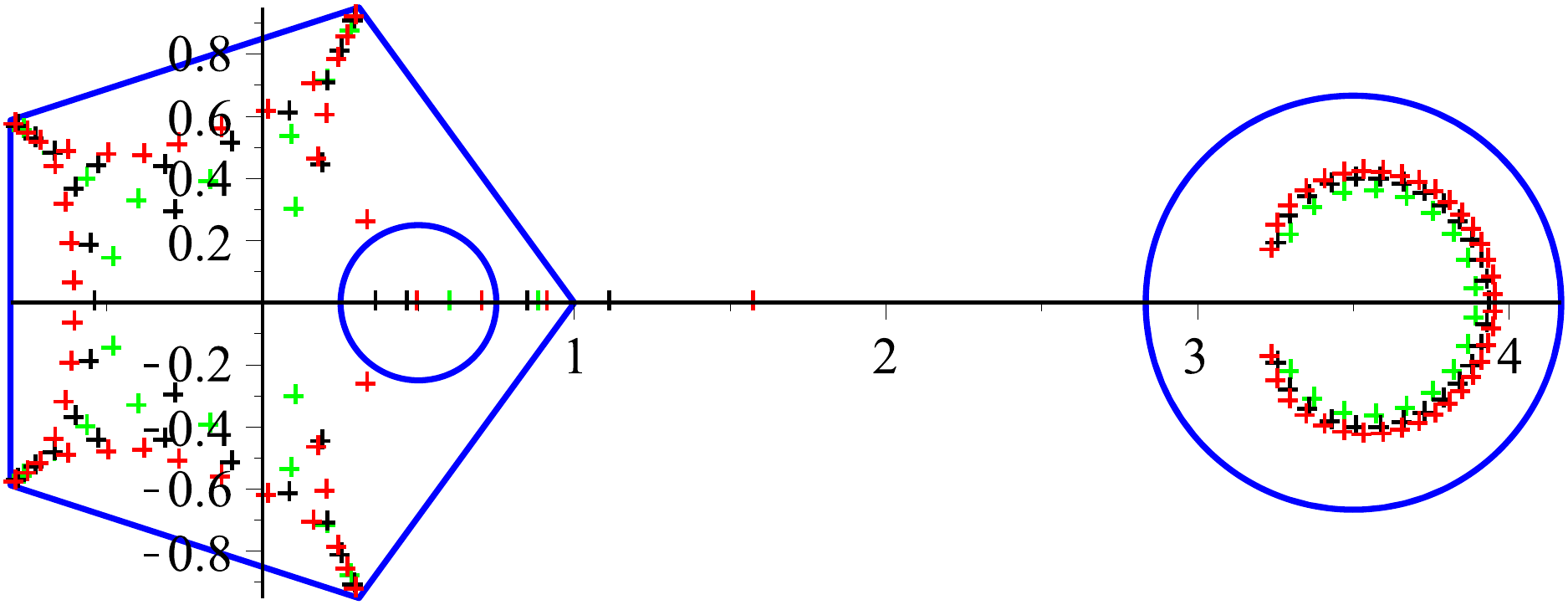}}
\end{center}
\end{minipage}
\caption{Plots of the zeros of $p_n(\mu_0,z)$, for $n=40,60$ and $80$ (above)
and $p_n(\mu_1,z)$, for $n=40,60$ and $80$ (below), related to Example~\ref{Ex1}.}
\label{fig1}
\end{figure}

In Figure~\ref{fig1} we plot the zeros of $p_n(\mu_0,z)$  and of $p_n(\mu_1,z)$, for $n=40,60$
and $80$. The close resemblance of the zeros in the two plots suggested that the behaviour of the associated orthonormal
polynomial sequence outside $G$ should be essentially the same, which is indeed what
Theorem~\ref{thm:main-gen} asserts. On the other hand, we wish to emphasize that, in general, the zeros of orthogonal polynomials
for PS perturbations of a measure $\mu_1$ need not have the same limit behaviour as the zeros of the orthogonal polynomials
generated by $\mu_1$; see Proposition~\ref{thm:Ex3} and Figure ~\ref{fig4} below. However,  the balayage of their limit measures
must be the same; see Theorem~\ref{th:ST4.7}.

The constructions of the orthonormal polynomials utilized for the plots in this paper were performed
by applying the Arnoldi variant of the Gram-Schmidt method; see e.g. \cite[Section 7.4]{St-CA13}.
All the computations  were carried out on a MacBook Pro using Maple.

Essential to the proof of Theorem~\ref{thm:main-gen}
is the following simple lemma, which is of independent interest.
\begin{lemma}\label{lem:main1}
Let $\mu_0$ and $\mu_1$ be two measures having compact and infinite support  in $\mathbb{C}$, with
$\mu_0\ge\mu_1$. Then for all $n\in\mathbb{N}\cup\{0\}$,
\begin{equation}\label{eq:lem-main1a}
\frac{\gamma_n(\mu_1)}{\gamma_n(\mu_0)}=1+\beta_n,
\end{equation}
where $\beta_n$ is non-negative and satisfies
\begin{equation}\label{eq:lem-main1b}
\frac{1}{\left(1-\|p_n(\mu_0,\cdot)\|^2_{L^2(\mu_2)}\right)^{1/2}}-1\le
\beta_n\le\left(1+\|p_n(\mu_1,\cdot)\|^2_{L^2(\mu_2)}\right)^{1/2}-1.
\end{equation}
Furthermore,
\begin{equation}\label{eq:compare}
\|p_n(\mu_0,\cdot)\|_{L^2(\mu_2)}\le\|p_n(\mu_1,\cdot)\|_{L^2(\mu_2)}
\end{equation}
and
\begin{equation}\label{eq:corbetan-1}
\|p_n(\mu_0,\cdot)-p_n(\mu_1,\cdot)\|^2_{L^2(\mu_1)}\le 2\beta_n.
\end{equation}
Finally, for any $z\in\overline{\mathbb{C}}\setminus{\rm Co}(S_1)$,
\begin{equation}\label{eq:ratioout}
\left|\frac{p_n(\mu_0,z)}{p_n(\mu_1,z)}-1\right|\le
\sqrt{2\beta_n}\left[1+\frac{\diam(S_1)}{\dist(z,\textup{Co}(S_1))}\right]^2.
\end{equation}
\end{lemma}

If $\displaystyle{\lim_{n\to\infty}\gamma_n(\mu_1)^{1/n}=1/\capy(S_1)}$, then $\mu_1$ belongs to the important
class of measures \textbf{Reg} investigated in \cite{StTo}. Furthermore, if in a neighbourhood of infinity
\begin{equation}\label{eq:def-Rat}
\lim_{n\to\infty}\frac{p_n(\mu_1,z)}{p_{n+1}(\mu_1,z)}=f(z),
\end{equation}
for some analytic function $f(z)$, then we say $\mu_1$ belongs to the class
of measures \textbf{Ratio}($f$). The following result shows that both classes of measures are invariant under
PS perturbations.
\begin{proposition}\label{cor:inva}
Let $\mu_1$ be in the class \textbf{Reg}, respectively \textbf{Ratio}($f$). If $\mu_0$ is a PS perturbation of
$\mu_1$, then $\mu_0$ is in the class \textbf{Reg}, respectively \textbf{Ratio($f$)}.
\end{proposition}

We refer to \cite{BLS}, \cite{KhSt}, \cite{PuSt}, \cite{SaSt2012}, \cite{Si12} and \cite{Si14}, for a list of
recent applications of measures in the class \textbf{Ratio}($f$), where Proposition~\ref{cor:inva}
is expected to have an impact; namely results concerning measures in the class
\textbf{Ratio}($f$) should  be easily extended to hold for PS perturbations of these measures.

As an illustration, we consider the infinite upper Hessenberg matrix
$M_{\mu_0}$ associated with the orthonormal sequence $\{p_n(\mu_0,z)\}_{n=0}^\infty$;
that is, $M_{\mu_0}:=[b_{k,j}]_{k,j=1}^\infty$, where
$b_{k,j}=\langle zp_j(\mu_0,\cdot),p_k(\mu_0,\cdot)\rangle_{\mu_0}$; see, e.g., \cite{SaSt2012} and
\cite{Si14}.
Then, the following result is an easy consequence of Theorem~\ref{thm:main-gen}, Proposition~\ref{cor:inva} and
\cite[Corollary~1.4]{Si14}.

\begin{proposition}\label{cor:sima}
Assume that $\mu_1\in$ \textbf{Ratio($f$)} satisfies
$$
\liminf_{n\to\infty}\frac{\gamma_n(\mu_1)}{\gamma_{n+1}(\mu_1)}>0.
$$
If $\mu_0$ is a PS perturbation of $\mu_1$, then the matrix $M_{\mu_0}$ is
weakly asymptotically Toeplitz; i.e. for all $k\ge -1$,  $\lim_{n\to\infty}b_{n-k,n}$ exists.
\end{proposition}

In the following proposition we give an example where sharp rates of convergence are obtained
for the assertions (i), (ii) and (iii) of Theorem ~\ref{thm:main-gen}.

\begin{proposition}\label{thm:disk}
Let $\mathcal{K}$ be a compact subset of the unit disk $\mathbb{D}$ and let $r:=\max\{|z|:z\in\mathcal{K}\}$.
If $\mu_1:=A|_{\mathbb{D}\setminus\mathcal{K}}$ and $\mu_0:=A|_{\mathbb{D}}$, then the following hold:
\begin{itemize}
\item[(i)]
$$
\frac{\gamma_n(\mu_1)}{\gamma_n(\mu_0)}=\frac{\gamma_n(\mu_1)}{\sqrt{\frac{n+1}{\pi}}}=
1+O\left(r^{2n}\right),
$$
and the order $O\left(r^{2n}\right)$ is sharp;
\item[(ii)]
$$
\|p_n(\mu_1,\cdot)-p_n(\mu_0,\cdot)\|_{L^2(\mathbb{D})}=
\left\|p_n(\mu_1,\cdot)-{\sqrt{\frac{n+1}{\pi}}}z^n\right\|_{L^2(\mathbb{D})}
=O\left(r^{n}\right);
$$
\item[(iii)]
$$
\frac{p_n(\mu_1,z)}{p_n(\mu_0,z)}=
\frac{p_n(\mu_1,z)}{\sqrt{\frac{n+1}{\pi}}z^n}
=1+O\left(r^n\right),\quad |z|>1;
$$
\item[(iv)]
$$
\max_{z\in\overline{\mathbb{D}}}\left|p_n(\mu_1,z)-
\sqrt{\frac{n+1}{\pi}}z^n\right|=O\left(nr^n\right).
$$
\end{itemize}
In the above, $O$ depends on $r$ only.
\end{proposition}

Next we state a proposition that is needed in establishing the last assertion in Theorem~\ref{thm:main-gen}.
It is a generalisation of assertion (2.10) in \cite{SSST} but we provide a shorter proof.
\begin{proposition}\label{lem:diagkernelout}
If $\mathcal{K}$ is a compact subset of $\Omega$, then
\begin{equation}\label{eq:diagkernelout}
\sum_{n=0}^\infty |p_n(\mu_1,z)|^2=\infty
\end{equation}
or, equivalently,
\begin{equation}\label{eq:chrfout}
\lim_{n\to\infty}\lambda_n(\mu_1,z)=0,
\end{equation}
uniformly for $z\in \mathcal{K}$.
\end{proposition}


\section{Zeros of PS perturbations of Bergman polynomials}

Our goal in this section is to provide results concerning the asymptotic behaviour
of zeros of PS perturbations of  Bergman polynomials. (By the term \textit{Bergman polynomials} we mean
polynomials orthonormal with respect to the area measure on bounded regions.)

To do this it will be convenient to use the notations introduced in \cite{SSST} regarding
\emph{archipelaga with lakes}.
More precisely, let $G:=\bigcup_{j=1}^m G_j$ be a finite union of bounded Jordan domains
$G_j$, $j=1,\ldots,m$, in the complex plane $\mathbb{C}$, with pairwise disjoint closures, let
$\mathcal{K}$ be a compact subset of $G$ and
consider the set $G^\star$ obtained from $G$ by removing $\mathcal{K}$; i.e.,
$G^\star:=G\setminus \mathcal{K}$.  Set $\Gamma_j:=\partial G_j$ for the  respective
boundaries and let $\Gamma:=\cup_{j=1}^m \Gamma_j$ denote the boundary of $G$.
Let $\Omega$ denote the unbounded component of $\overline{\mathbb{C}}\setminus \overline{G^\star}$, so that here
$\Omega=\overline{\C}\setminus\overline{G}$; see Figure~\ref{fig:Archi-lakes}. Note that
$\Gamma=\partial G=\partial\Omega$. We refer to $G$ as an archipelago and to $G^\star$ as an archipelago with
lakes.

In this case, $\mu_0=A|_{G}$ and $\mu_1=A|_{G^\star}$, so that  according to Lemma~2.2
in \cite{SSST}, $\mu_0$ is a PS perturbation of $\mu_1$. Further, from Lemma~3.2 in \cite{GPSS},
$\mu_0$ belongs to the class of measures $\textbf{Reg}$ and, in view of Theorem~\ref{thm:main-gen},
so does $\mu_1$.

Here, we will use the notation
$p_n(G,z)$ and $p_n(G^\star,z)$, in place of $p_n(\mu_0,z)$ and $p_n(\mu_1,z)$,
for the corresponding Bergman polynomials.
Also, we will use
$\gamma_n(G)$, $\gamma_n(G^\star)$,
$\|\cdot\|_{L^2(G)}$ and
$\|\cdot\|_{L^2(G^\star)}$, in the place of  $\gamma_n(\mu_0)$,
$\gamma_n(\mu_1)$,  $\|\cdot\|_{L^2(\mu_0)}$, and $\|\cdot\|_{L^2(\mu_1)}$, respectively.

\begin{figure}[t]
\begin{center}
\includegraphics[scale=0.90]{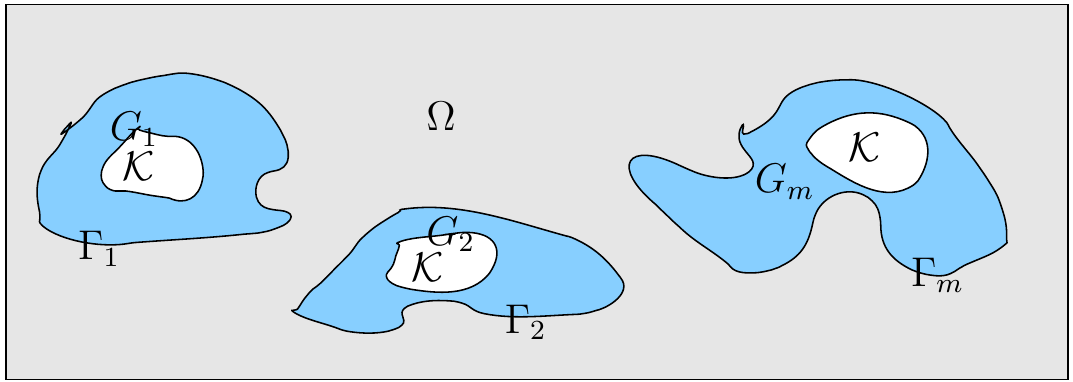}
\caption{Archipelagos with lakes}
\label{fig:Archi-lakes}
\end{center}
\end{figure}

The main task of this section is to describe the asymptotic behaviour of the zeros of the
polynomials $p_n(G^\star,z)$  under various assumptions on the boundary $\Gamma_j$ of each
individual island $G_j$. The behaviour of the zeros of $p_n(G,z)$ was investigated in \cite{GPSS}.

Our main tool is the \textit{normalized counting measure} $\nu_n$ for the zeros of a the
Bergman polynomial $p_n(G^\star,z)$; that is,
\begin{equation}\label{eq:nu}
\nu_n:=\frac{1}{n}\sum_{p_n(G^\star,z)=0}\delta_z,
\end{equation}
where $\delta_z$ is the unit point mass (Dirac delta) at the point $z$.
Note that, since all the zeros of $p_n(G^\star,z)$ lie in the convex hull of $G$, it follows from
Helly's selection theorem that the sequence of measures $\{\nu_n\}_{n=1}^\infty$
has convergent subsequences with limit measures supported on $\textup{Co}(G)$.

With $E$ a compact set in the complex plane of positive capacity, we denote by
$\mu_E$ the \textit{equilibrium measure}
(energy minimizing Borel probability measure on $E$) for the logarithmic potential on $E$; see e.g.,
\cite[Chapter 3]{Ra} and
\cite[Section I.1]{ST}. As is well-known, the support $\mathrm{supp}(\mu_E)$ of $\mu_E$ lies on the
boundary  of the unbounded component of $\overline{\mathbb{C}}\setminus E$.

Our first result is a consequence of the fact that $A|_{G^\star}$ is a measure in the class \textbf{Reg}.
Its proof is given in Section~4.
\begin{theorem}\label{th:ST4.7}
If $\mu$ is any weak-star limit measure of the sequence $\{\nu_n\}_{n\in\mathbb{N}}$ in \eqref{eq:nu}, then $\mu$ is a Borel
probability measure supported on $\overline{G}$ and
$\mu^b=\mu_\Gamma$, where $\mu^b$ is the balayage of $\mu$ out of $\overline{G}$ onto
$\partial\Omega$.
Similarly, the sequence of balayaged counting measures converges to $\mu_E$:
\begin{equation}
\nu^b_n\,{\stackrel{*}{\longrightarrow}}\, \mu_\Gamma,\quad n\to\infty,\quad n\in\mathbb{N}.
\end{equation}
\end{theorem}
By the weak-star convergence of a sequence of measures $\tau_n$ to a measure $\tau$ we mean that,
for any continuous $f$ with compact support in $\mathbb{C},$ there holds
\begin{equation*}
\int f d\tau_n \to \int fd\tau, \quad\textup{as }n\to\infty.
\end{equation*}
For properties of balayage  see \cite[Sect. II.4]{ST}.

The next result illustrates that the zero behaviour of polynomials orthogonal with respect to a PS
perturbation of a measure $\mu_1$ can have an asymptotic limit that is different from that of the zeros of
orthogonal polynomials with respect to $\mu_1$. However, as Theorem~\ref{th:ST4.7} asserts, the balayage onto the
boundary of these two limiting distributions of zeros must be the same.

Let $G=\mathbb{D}$, the unit disk, and $\mathcal{K}:=\{z:|z-a|\le \varrho\}$, $|a|+\varrho<1$, $\varrho>0$.
In order to describe the zero behaviour of the sequence $\{p_n(G^\star,z)\}_{n=0}^\infty$, $G^\star=\mathbb{D}\setminus\mathcal{K}$,
we recall that there exists a unique pair of points $z_1$ and $z_2$ that are 
 mutually inverse points with respect to the two circles
$\mathbb{T}:=\partial\mathbb{D}$ and $\{z:|z-a|=\varrho\}$, that is
\begin{equation}\label{eq:zeta12}
z_1\overline{z_2}=1 \quad\textup{and}\quad
(z_1-a)(\overline{z_2-a})=\varrho^2.
\end{equation} 
Let $z_1$ denote the point that lies in $\mathcal{K}$ ($z_2$ will then lie outside $\mathbb{D}$).
We prove in Section~\ref{Proofs} the following.
\begin{proposition}\label{thm:Ex3}\footnote{This result 
was presented by the first author at the 2011 conference\lq\lq Computational
Complex Analysis and Approximation Theory" held in Protaras, Cyprus in honor of Nick Papamichael.}
With the above notation, there exists a subsequence $\mathcal{N}\subset\mathbb{N}$ such that
the normalized zero counting measures for $p_n(G^\star,z)$ satisfy 
\begin{equation}\label{eq:cic}
\nu_{n}\,{\stackrel{*}{\longrightarrow}}\,
\mu_{|z_1|},\quad n\to\infty,\quad n\in\mathcal{N},
\end{equation}
where $\mu_{|z_1|}$ denotes the normalized arclength measure on the circle $|z|=|z_1|$.
\end{proposition}
Thus, no matter what the relative position of $\mathcal{K}$, a weak limit of
$\nu_{n}$ will invariably be the arclength measure on a specific circle in $\mathbb{D}$, always centered
at the origin.

By way of illustration, in Figure~\ref{fig4} we plot the zeros of $p_n(G^\star,z)$, for $n=120,140$ and $160$, for
$a=0.2$, $a=0.4$ and $\varrho=0.25$.

\begin{figure}[t]
\begin{center}
\fbox{\includegraphics[scale=0.30]{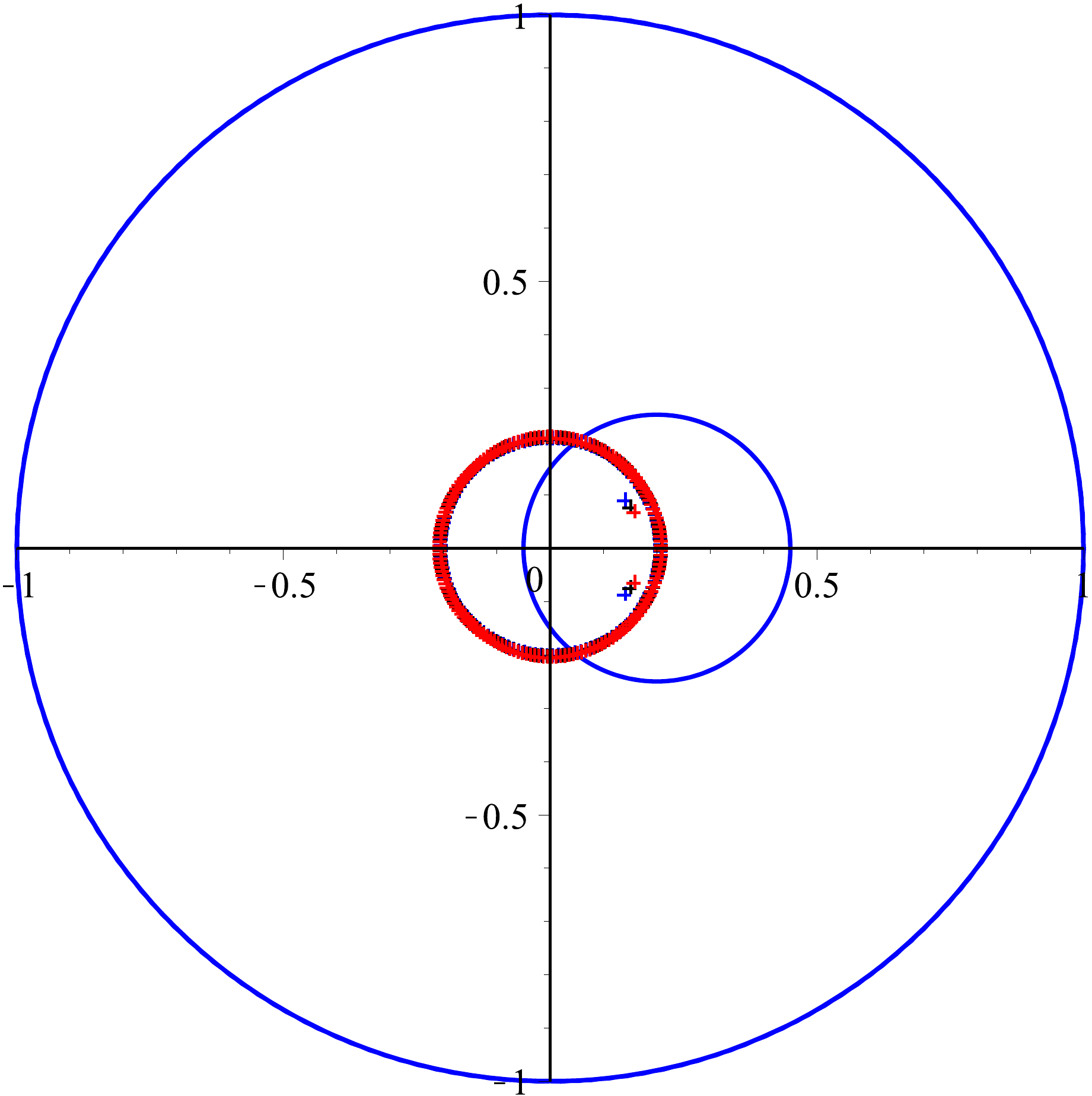}}
\fbox{\includegraphics[scale=0.30]{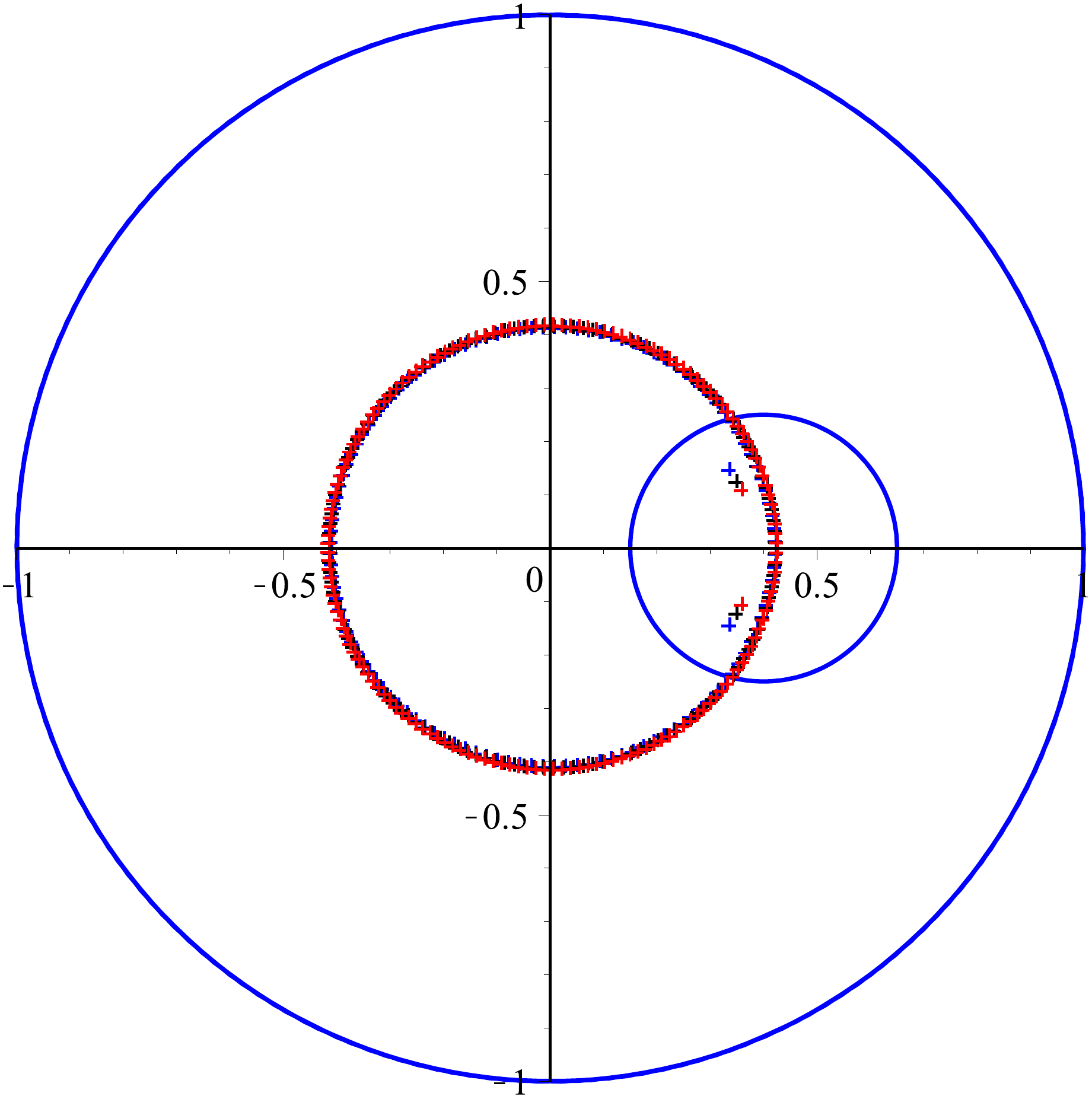}}
\end{center}
\caption{Plots of the zeros of $p_n(G^\star,z)$, for $n=120,140$ and $160$, related to Proposition~\ref{thm:Ex3},
with $a=0.2$ (left) $a=0.4$ (right) and $\varrho=0.25$.}
\label{fig4}
\end{figure}

\medskip
Below we need the notion of an \emph{inward-corner} (IC) point. Following
\cite{SaSt2015}, we say that a point $z_0$ on the boundary $\Gamma_j$ of $G_j$ is an IC point
if there exists a circular sector of the form
$S:=\{z:0< |z-z_0|< r,\,\alpha\pi<\textup{arg}(z-z_0)<\beta\pi\}$
with $\beta-\alpha > 1$  whose closure is contained in $G_j$ except for $z_0$.

The following theorem is an immediate consequence of Corollary 2.2 in \cite{SaSt2015}; cf. the
proof of Theorem~\ref{th:ST4.7} in Section 4.

\begin{theorem}\label{th:ic}
With the notation above, assume that for each $j=1,\ldots,l$ the boundary $\Gamma_j$ of $G_j$ contains
an IC point\footnote{The ordering of $G_j$'s is irrelevant here.}. Then, with $\nu_n$ as in \eqref{eq:nu},
\begin{equation}\label{eq:ic}
\nu_{n}|_{\mathcal{V}}\,{\stackrel{*}{\longrightarrow}}\,
\mu_\Gamma|_{\mathcal{V}},\quad n\to\infty,\quad n\in\mathbb{N},
\end{equation}
where $\mathcal{V}$ is an open set containing $\bigcup_{j=1}^l\overline{G}_j$, such that if $l<m$
the distance of $\overline{\mathcal{V}}$ from $\bigcup_{j=l+1}^m \overline{G}_j$ is positive.
\end{theorem}
The result of the theorem remains valid under the weaker condition that
the boundary $G_j$, for each $j=1,\ldots,l$, contains a
\emph{non-convex type singularity}; see \cite{SaSt2015} for details.

Theorem~\ref{th:ic} is illustrated by the following example.

\begin{example}\label{ex:sector-disk}
Consider the case $G=G_1\cup G_2$, where $G_1$ is a circular sector of opening angle $3\pi/2$ and radius
$1$ with center at the origin and $G_2$ is the disk with center at $3.5$ and radius $2/3$ and let
$\mathcal{K}$ be the closed disk inside $G_1$ with center at $1/2$ and radius $1/4$.
\end{example}

\begin{figure}[t]
\begin{center}
\fbox{\includegraphics[scale=0.50]{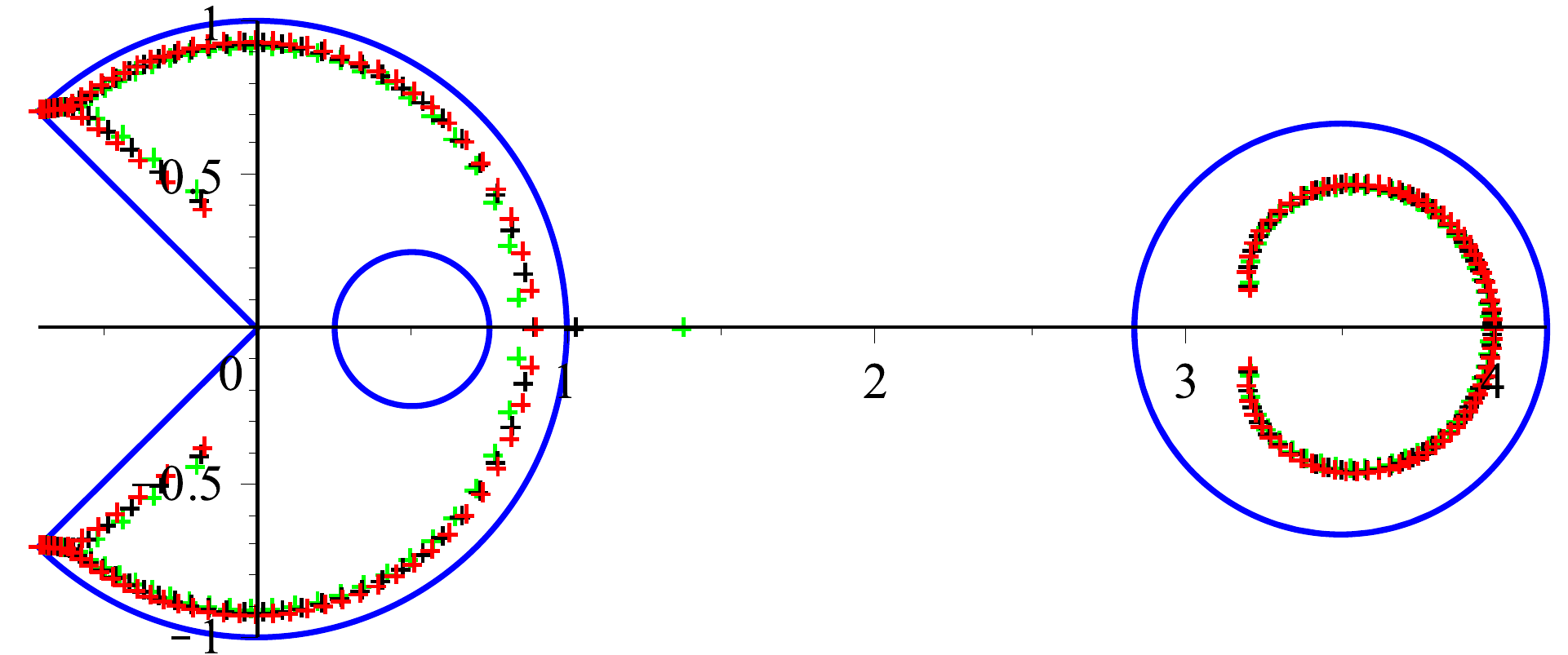}}
\end{center}
\caption{Plots of the zeros of $p_n(G^\star,z)$, for $n=120,140$ and $160$ related to Example~\ref{ex:sector-disk}.}
\label{fig3}
\end{figure}

In Figure~\ref{fig3} we plot the zeros of $p_n(G^\star,z)$, with $G^\star=G\setminus\mathcal{K}$, for $n=120,140$ and $160$.
According to Theorem~\ref{th:ic},
$\nu_{n}|_{\mathcal{V}}\,{\stackrel{*}{\longrightarrow}}\,\mu_\Gamma|_{\mathcal{V}}$, as $n\to\infty$, through
the whole sequence $\mathbb{N}$, where $\mathcal{V}$ is some open neighbourhood of $\overline{G}_1$.
This, in particular, implies that every point of the boundary of $G_1$
will attract zeros of $p_n(G^\star,z)$, a fact which is corroborated by the position of zeros in the plot.
The obvious reluctance of the zeros to
approach the reentrant corner of the sector can be accounted for by the fact that the equilibrium measure at
this corner has zero density.


We end this section by noting that there is a recent growing interest for asymptotics of Bergman
polynomials with weights supported
on the unit disk; see, e.g. \cite{Sa10}, \cite{Mina2013} and the references therein.

It is evident that much of our results above hold true if the Bergman polynomials $p_n(G^\star,z)$ are
replaced by  weighted Bergman polynomials, defined by the inner product
\begin{equation*}
\langle f,g\rangle_{Gw} := \int_G f(z) \overline{g(z)} w(z)dA(z),
\end{equation*}
for some reasonable weight function $w\in L^2(G)$. See, for example, the book by
Suetin \cite[Chapter~II]{Su74}
and the more recent results of Totik in \cite[Theorems 1.3 \& 1.5]{To09}.
Clearly, the polynomials $\{p_n(G^\star,z)\}$ correspond to the characteristic function
\begin{equation*}
 w(z)=\left\{
\begin{array}{cl}
0,  &\textup{if}\ z\in\mathcal{K},\\
1,  &\textup{otherwise}.
\end{array}
\right.
\end{equation*}

\section{Proofs}\label{Proofs}
This section contains the proofs of  results in the preceding  sections.

\subsection{Proof of Proposition~\ref{prop:Co}}
We shall show that there are no points of $S_2$ in $\Omega=\overline{\mathbb{C}}\setminus\textup{Pc}(S_1)$.
Then, since $S_0=S_1\cup S_2$ we have $S_0\subset\textup{Pc}(S_1)$.
Furthermore, strict inequality cannot hold for the capacities, since ${\capy(S_2\cap\Omega)=0}$; thus $\capy(S_1)=\capy(S_0)$.

To see that there are no points of $S_2$ in $\Omega$ first observe that there cannot be
any isolated points of $S_2$ in
$\Omega$, since at any such point, say $z_0$, we would have from \eqref{eq:main-gen-asu1},
$$
0=\lim_{n\to\infty}\int|p_n(\mu_1,z)|^2d\mu_2(z)\ge\lim_{n\to\infty}|p_n(\mu_1,z_0)|^2\mu_2(\{z_0\});
$$
but by a theorem of Ambroladze \cite[Theorem~1]{Am95}, if $g_\Omega(z,\infty)$ denotes the \textit{Green function} with pole
at infinity associated with $\Omega$, then
$$
\limsup_{n\to\infty}|p_n(\mu_1,z_0)|^{1/n}\ge\exp\{g_\Omega(z_0,\infty)\}>1,
$$
which yields a contradiction. (If the complement  $\overline{\mathbb{C}}\setminus\Omega$ of $\Omega$ has capacity zero, then
$g_\Omega(z,\infty)\equiv\infty$;  see \cite[Appendix V]{StTobo}.)

Next, suppose that $z_0\in S_2\cap\Omega$ and $z_0$ is an accumulation point of $S_2$. Let
$\overline{B}(z_0)\subset\Omega$ be a closed disk centered at $z_0$. By a result of Widom \cite{Wi67}; see also
\cite[Theorem~2.1.1]{StTobo},
there exists an integer $L$ such that each $p_n(\mu_1,z)$ has at most $L$ zeros in $\overline{B}(z_0)$.
Let
$$
q_n(z):=\prod(z-\zeta),\quad \zeta\in\overline{B}(z_0),\quad p_n(\mu_1,\zeta)=0,
$$
and consider a subsequence $\mathcal{N}\subset\mathbb{N}$ such that $q_n(z)$ converges locally
uniformly in $\mathbb{C}$ to a monic polynomial $Q(z)$ of degree at most $L$.
By Corollary~1.1.5 in \cite{StTobo}
\begin{equation}\label{Cor115:SaTobo}
\liminf_{n\to\infty}\left|\frac{p_n(\mu_1,z)}{q_n(z)}\right|^{1/n}\ge\exp\{g_\Omega(z,\infty)\},
\end{equation}
uniformly on $\overline{B_1}(z_0)$, where $\overline{B_1}(z_0)$ is a closed disk centered at $z_0$ and
contained in ${B}(z_0)$. Then, there exists a constant $\varrho>1$ such that
\begin{align*}
0=&\lim_{n\to\infty}\int|p_n(\mu_1,z)|^2d\mu_2(z)\ge\lim_{n\to\infty}\int_{\overline{B_1}(z_0)}|p_n(\mu_1,z)|^2d\mu_2(z)\\
 \ge& C\varrho^n\lim_{n\to\infty,n\in\mathcal{N}}\int_{\overline{B_1}(z_0)}|q_n(\mu_1,z)|^2d\mu_2(z).
\end{align*}
Hence,
$$
0=\lim_{n\to\infty,n\in\mathcal{N}}\int_{\overline{B_1}(z_0)}|q_n(\mu_1,z)|^2d\mu_2(z)=
\int_{\overline{B_1}(z_0)}|Q(z)|^2d\mu_2(z).
$$
This yields a contradiction, since $\mu_2$ has infinitely many points of support in $\overline{B_1}(z_0)$ where $Q(z)$ must be zero.
\qed

\subsection{Proof of Proposition~\ref{prop1}}
Let $\{p_n(z)\}_{n=0}^\infty$ denote the associated sequence of orthonormal polynomials
with respect to the arclength measure on $\Gamma$ and let $\Phi$ denote the conformal map
from the exterior of $\Gamma$ to the exterior of the unit circle, normalized by
$\Phi(\infty)=\infty$ and $\Phi^\prime(\infty)>0$. Then, from \cite[Theorem 2.3]{Su66b}
we have
\begin{equation}\label{eq:Suetin-asy}
p_n(z)=\Phi^n(z)\sqrt{\Phi^\prime(z)}\left[1+O\left(1/n\right)\right],
\end{equation}
uniformly for $z\in\Gamma$. Our assumption on $\Gamma$ implies that $\Phi^\prime$ is finite
and non-vanishing on $\Gamma$ and that $\Phi^{\prime\prime}$ is bounded on $\Gamma$; see e.g.
\cite[Sections 3.3--3.4]{Po92}.

Let $z_0\in\Gamma$ be fixed. Then, from integration along $\Gamma$
we have
\begin{align*}
\int_{z_0}^z p_n(\zeta)d\zeta &= \int_{z_0}^z
\Phi^n(\zeta)\frac{\Phi^\prime(\zeta)}{\sqrt{\Phi^\prime(\zeta)}}\left[1+O\left(1/n\right)\right]d\zeta\\
&= \int_{z_0}^z
\Phi^n(\zeta)\frac{\Phi^\prime(\zeta)}{\sqrt{\Phi^\prime(\zeta)}}d\zeta+O\left(1/n\right),
\quad z\in\Gamma.
\end{align*}
Integrating by parts gives uniformly for $z\in\Gamma$,
\begin{align}\label{eq:Qp}
Q_{n+1}(z)&:=\int_{z_0}^z p_n(\zeta)d\zeta=
\frac{\Phi^{n+1}(z)}{n+1}\frac{1}{\sqrt{\Phi^\prime(z)}}-
\frac{\Phi^{n+1}(z_0)}{n+1}\frac{1}{\sqrt{\Phi^\prime(z_0)}}\nonumber\\
&+\frac{1}{2}\int_{z_0}^z
\frac{\Phi^{n+1}(\zeta)}{n+1}\frac{\Phi^{\prime\prime}(\zeta)}
{[\Phi^{\prime}(\zeta)]^{3/2}}d\zeta+O\left(1/n\right)=O\left(1/n\right).
\end{align}
Now by using Green\rq s formula (see e.g. \cite[p.~10]{Gabook87}) and the Cauchy-Schwarz inequality we have
\begin{align*}
\int_G|p_n(z)|^2dA(z)&=\frac{1}{2i}\int_\Gamma p_n(z)\overline{Q_{n+1}}(z)dz
\le\frac{1}{2}\int_\Gamma |p_n(z)||Q_{n+1}(z)|ds\\
&\le \frac{1}{2}\left[\int_\Gamma |p_n(z)|^2ds\right]^{1/2}\left[\int_\Gamma |Q_{n+1}(z)|^2
ds\right]^{1/2}
\end{align*}
and the required result follows from (\ref{eq:Qp}) because of the normality of the polynomials $p_n$.
\qed

\subsection{Proof of Lemma~\ref{lem:main1}}
We aim first at relating  the leading coefficients $\gamma_n(\mu_0)$ and $\gamma_n(\mu_1)$.
To this end, we note that
\begin{equation}\label{eq:comp-pri-gamma}
\gamma_n(\mu_0)\le \gamma_n(\mu_1),\quad n=0,1,\ldots,
\end{equation}
which is a simple consequence of the minimal property (\ref{eq:minimal1}).

The fact that $\beta_n\ge 0$ is evident from the inequality (\ref{eq:comp-pri-gamma}).
To prove (\ref{eq:lem-main1b}) we note first the two obvious relations
\begin{equation}\label{eq:pnG=1-}
\|p_n(\mu_0,\cdot)\|^2_{L^2(\mu_1)}=1-\|p_n(\mu_0,\cdot)\|^2_{L^2(\mu_2)}
\end{equation}
and
\begin{equation}\label{eq:pnG*=1+}
\|p_n(\mu_1,\cdot)\|^2_{L^2(\mu_0)}=1+\|p_n(\mu_1,\cdot)\|^2_{L^2(\mu_2)}.
\end{equation}
Since $\mu_1$ has infinite support, \eqref{eq:pnG=1-} shows that
$$
\|p_n(\mu_0,\cdot)\|_{L^2(\mu_2)}<1.
$$

Furthermore, from  the Parseval identity we have
\begin{align}\label{eq:ParId1}
\|p_n(\mu_1,\cdot)\|^2_{L^2(\mu_0)}&=\sum_{k=0}^n|\langle p_n(\mu_1,\cdot),p_k(\mu_0,\cdot) \rangle_{\mu_0} |^2\nonumber \\
&\ge|\langle p_n(\mu_1,\cdot),p_n(\mu_0,\cdot) \rangle_{\mu_0} |^2,
\end{align}
and
\begin{align}\label{eq:ParId2}
\|p_n(\mu_0,\cdot)\|^2_{L^2(\mu_1)}&=\sum_{k=0}^n|\langle p_n(\mu_0,\cdot),p_k(\mu_1,\cdot) \rangle_{\mu_1} |^2\nonumber\\
&\ge |\langle p_n(\mu_0,\cdot),p_n(\mu_1,\cdot) \rangle_{\mu_1} |^2.
\end{align}

The required upper and lower estimates in (\ref{eq:lem-main1b}) then follow easily from
(\ref{eq:pnG=1-})--(\ref{eq:ParId2}), because
\begin{equation}\label{eq:<pn*,pn>}
\begin{alignedat}{1}
\langle p_n(\mu_1,\cdot),p_n(\mu_0,\cdot)\rangle_{\mu_0}&=\frac{\gamma_n(\mu_1)}{\gamma_n(\mu_0)},
\\
\langle p_n(\mu_0,\cdot),p_n(\mu_1,\cdot)\rangle_{\mu_1}&=\frac{\gamma_n(\mu_0)}{\gamma_n(\mu_1)}.
\end{alignedat}
\end{equation}
Inequality \eqref{eq:compare} follows at once from  \eqref{eq:lem-main1b}.

Next we note the following two relations, which hold for all $n\in\mathbb{N}\cup\{0\}$:
\begin{equation}\label{eq:lem-main2}
\|p_n(\mu_0,\cdot)-p_n(\mu_1,\cdot)\|^2_{L^2(\mu_0)}=\|p_n(\mu_1,\cdot)\|^2_{L^2(\mu_2)}-2\beta_n
\end{equation}
and
\begin{equation}\label{eq:lem-main2-2}
\|p_n(\mu_0,\cdot)-p_n(\mu_1,\cdot)\|^2_{L^2(\mu_1)}=
\frac{2\beta_n}{1+\beta_n}-\|p_n(\mu_0,\cdot)\|^2_{L^2(\mu_2)}.
\end{equation}
These can be readily seen by expanding $\|p_n(\mu_0,\cdot)-p_n(\mu_1,\cdot)\|^2_{L^2(\mu_0)}$ and
$\|p_n(\mu_0,\cdot)-p_n(\mu_1,\cdot)\|^2_{L^2(\mu_1)}$,
together with the definition of $\beta_n$ in
(\ref{eq:lem-main1a}) and the four relations in (\ref{eq:pnG=1-}), (\ref{eq:pnG*=1+}) and
(\ref{eq:<pn*,pn>}).
Then, the inequality in \eqref{eq:corbetan-1} is immediate from (\ref{eq:lem-main2-2}),
because $\beta_n\ge 0$.

Finally, \eqref{eq:ratioout} follows from \eqref{eq:corbetan-1}  by working as in obtaining
(d) from (b) in the proof of Theorem 2.1 in \cite{SSST}, because $\textup{Co}(S_0)=\textup{Co}(S_1)$.
\qed

\subsection{Proof of Proposition~\ref{lem:diagkernelout}}
Let $\zeta\in \mathcal{K}$ be fixed. By Runge's theorem there exists a sequence of polynomials $q_n(z)=q_n(z;\zeta)$
such that $q_n(z)$ converges uniformly to zero on the boundary of $\Omega$ (and hence on $S_1$) and also converges to one on a sufficiently small closed disk $\overline{D}(\zeta,\epsilon)\subset\Omega$
centered at $\zeta$. But then it is immediate from the extremal property of Christoffel functions that
$\lambda_n(\mu_1,z)\to 0$, uniformly on $\overline{D}(\zeta,\epsilon)$. Since $\mathcal{K}$ is compact,
it can be covered by finitely many such disks, and the assertion of the proposition follows.
\qed

\subsection{Proof of Theorem~\ref{thm:main-gen}}
Our reasoning is guided by the arguments for Bergman polynomials in \cite{SSST}.

Assertions (i) and (ii) are immediate from Lemma~\ref{lem:main1} and \eqref{eq:lem-main2}.

Assertion (iii) for  $z\in\overline{\mathbb{C}}\setminus{\rm Co}(S_0)$  follows from \eqref{eq:ratioout},
because $\textup{Co}(S_0)=\textup{Co}(S_1)$.
The claim for $z\in\Omega$ follows from the corresponding result in
$\overline{\mathbb{C}}\setminus{\rm Co}(S_0)$ by using a standard normal family argument.

Next we consider the sequence of numbers
\begin{equation}\label{eq:vare_m}
\varepsilon_m:=\sum_{j=m}^\infty \|p_j(\mu_1,\cdot)\|_{L^2(\mu_2)}^2, \quad m\in\mathbb{N},
\end{equation}
which by the assumption in \eqref{eq:main-gen-asu2} tends to zero as $m\to\infty$.
In contrast, from Proposition~\ref{lem:diagkernelout} it follows that for any given compact set
$\mathcal{K}\subset\Omega$ and positive integers $M$ and $m$,
\begin{equation}\label{eq:pre-lemggs}
\sum_{j=m}^{n}|p_j(\mu_1,z)|^2>M,\quad z\in \mathcal{K},
\end{equation}
for all sufficiently large $n$.

The next result provides an estimate for the relation between $\lambda_n(\mu_0,z)$ and $\lambda_n(\mu_1,z)$
in terms of $\varepsilon_m$ and the sum in (\ref{eq:pre-lemggs}).
\begin{lemma}\label{lem:ggs}
With the notation above, 
choose $n>m$ sufficiently large
 so that \textup{(\ref{eq:pre-lemggs})} holds. Then, for $\zeta\in\mathcal{K}$,
\begin{equation}\label{eq:ggs-1}
\lambda_n(\mu_1,z)\le\lambda_n(\mu_0,z)\le\lambda_n(\mu_1,z)\{1+D_{n}(z)\},
\end{equation}
where
\begin{equation}\label{eq:ggs-2}
D_{n}(z):=\varepsilon_m+\frac{\left(1+\varepsilon_m\right)}{M}
\sum_{j=0}^{m-1}|p_j(\mu_1,z)|^2.
\end{equation}
\end{lemma}

\begin{proof}
For $z\in\mathcal{K}$ consider the polynomial
\begin{equation}\label{eq:PnDef}
P_n(t):=\frac{\sum _{j=m}^n\overline{p_j(\mu_1,z)}p_j(\mu_1,t)}{\sum _{j=m}^n|p_j(\mu_1,z)|^2}.
\end{equation}
Since $P_n(z)=1$, we have from the minimal property (\ref{eq:Chrfun-def}) of Christoffel functions that
\begin{equation}\label{eq:lamGzle}
\lambda_n(\mu_0,z)\le \int|P_n(t)|^2d\mu_0(t).
\end{equation}

Furthermore, from the orthogonal expansion of $P_n$ with respect to $\mu_1$, we have
\begin{equation}\label{eq:PnoverGs}
\int|P_n(t)|^2d\mu_1(t) =\frac{1}{\sum _{j=m}^n|p_j(\mu_1,z)|^2}.
\end{equation}
By using the Cauchy-Schwarz  inequality for the sum in (\ref{eq:PnDef}) we get that
\begin{equation*}
\int|P_n(t)|^2d\mu_2(t)\le\frac{1}{\sum _{j=m}^n|p_j(\mu_1,z)|^2}\int \sum_{j=m}^n|p_j(\mu_1,t)|^2d\mu_2(t)
\end{equation*}
and, therefore,  we obtain from (\ref{eq:vare_m}) that
\begin{equation}\label{eq:PnoverB}
\int|P_n(t)|^2d\mu_2(t)\le \frac{\varepsilon_m}{\sum _{j=m}^n|p_j(\mu_1,z)|^2}.
\end{equation}

Now adding up (\ref{eq:PnoverGs}) and (\ref{eq:PnoverB}) we conclude, in view of (\ref{eq:lamGzle}), that
\begin{equation}
\lambda_n(\mu_0,z)\le \frac{1+\varepsilon_m}{\sum _{j=m}^n|p_j(\mu_1,z)|^2}.
\end{equation}
The upper estimate in (\ref{eq:ggs-1}) then follows from (\ref{eq:Chrfun-pro}),
(\ref{eq:PnoverB}) and \eqref{eq:pre-lemggs}, because
$$
\frac{1}{\sum _{j=m}^n|p_j(\mu_1,z)|^2}=\frac{1}{\sum _{j=0}^n|p_j(\mu_1,z)|^2}
\left[1+\frac{\sum_{j=0}^{m-1}|p_j(\mu_1,z)|^2}{\sum_{j=m}^{n}|p_j(\mu_1,z)|^2}\right].
$$
The lower estimate in \eqref{eq:ggs-1} is simply the comparison property (\ref{eq:comp-pri-lambda}).
\end{proof}

Finally, assertion \eqref{eq:main-gen-asu21} of Theorem~\ref{thm:main-gen} follows at once
from \eqref{eq:ggs-1}, because $\varepsilon_m$ can be made arbitrarily small
and $M$ can be chosen arbitrarily large.
\qed

\subsection{Proof of Proposition~\ref{cor:inva}}
When $\mu_1\in \textbf{Reg}$, the result is evident from Theorem~\ref{thm:main-gen}(i),
in view of Proposition~\ref{prop:Co}.
For $\mu_1\in \textbf{Ratio}(f)$, the result follows at once from Theorem~\ref{thm:main-gen}(iii), because
outside $\textup{Co}(S_0)=\textup{Co}(S_1)$,
$$
\frac{p_n(\mu_0,z)}{p_{n+1}(\mu_0,z)}=\frac{p_n(\mu_0,z)}{p_{n}(\mu_1,z)}\,
\frac{p_{n}(\mu_1,z)}{p_{n+1}(\mu_1,z)}\,\frac{p_{n+1}(\mu_1,z)}{p_{n+1}(\mu_0,z)};
$$
see also the remark following the statement of  Theorem~\ref{thm:main-gen}.
\qed

\medskip
Below we use $\|\cdot\|_{E}$ to denote the uniform norm over a compact set $E$.
We recall that $\Omega$ denotes the unbounded component of
$\overline{\mathbb{C}}\setminus S_1$.

\subsection{Proof of Proposition~\ref{thm:disk}}
Set $\widetilde{\mathcal{K}}:=\{z:|z|\le r\}$, $G:=\mathbb{D}$, $G^\star:=\mathbb{D}\setminus\mathcal{K}$ and consider
$\widetilde{G}:=G\setminus\widetilde{\mathcal{K}}$. Then it is easy to check
that
\begin{equation}\label{eq:pnD}
p_n(G,z)=\gamma_n(G)z^n\quad\mbox{ and }\quad
p_n(\widetilde{G},z)=\gamma_n(\widetilde{G})z^n
\end{equation}
where
\begin{equation}\label{eq:pnD}
\gamma_n(G)=\sqrt{\frac{n+1}{\pi}}\quad\mbox{ and }\quad
\gamma_n(\widetilde{G})=\sqrt{\frac{n+1}{\pi}}\left(1-r^{2n+2}\right)^{-1/2}.
\end{equation}
Since $\widetilde{G}\subset G^\star\subset G$, it follows from the minimal property (\ref{eq:minimal1}) that
$$
\gamma_n(G)\le\gamma_n(G^\star)\le\gamma_n(\widetilde{G}).
$$
Hence, (\ref{eq:pnD}) implies
$$
1\le\frac{\gamma_n(G^\star)}{\gamma_n(G)}=1+\beta_n,
$$
with $\beta_n=O(r^{2n})$, which yields (i). The sharpness claim is evident by letting
$\widetilde{\mathcal{K}}=\mathcal{K}$.

Next, by applying \eqref{eq:corbetan-1}  we get
$$
\|p_n(G^\star,\cdot)-p_n(G,\cdot)\|_{L^2(G^\star)}\le O(r^n),
$$
and (ii) follows because the norms $\|\cdot\|_{L^2(G)}$ and $\|\cdot\|_{L^2(G^\star)}$ are equivalent;
see, e.g., \cite[p.~2447]{SSST}.

Part (iii) then follows from an application of \eqref{eq:ratioout}.

Finally, the uniform estimate in part (iv) follows from the $L^2$-estimate in (ii) by using the  inequality
$$
\|P_n\|_{\ov G}\le Cn\|P_n  \|_{L^2(G)},
$$
which is valid for any polynomial $P_n$ of degree at most $n\in\mathbb{N}$, where the constant $C$
depends on $G$  only; see \cite[p. 38]{Su74}.
\qed


\subsection{Proof of Theorem~\ref{th:ST4.7}}
As was noted in Section~3,
$A|_{G^\star}$ belongs to the class $\textbf{Reg}$. Thus, from \cite{Sa90} and
\cite[Theorems~3.2.1 \& 3.2.3]{StTobo} we have
\begin{equation}\label{eq:SaTo}
\lim_{n\to\infty}\gamma_n(G^\star)^{1/n}=\frac{1}{\capy(\overline{G})}
\quad\mbox{ and }\quad
\lim_{n\to\infty}\|p_n(G^\star,\cdot)\|_{\overline{G}}^{1/n}=1,
\end{equation}
because $\capy(\overline{G})=\capy(G)=\capy(G^\star)$ and every point of $\Gamma$  is regular for the
Dirichlet problem in
$\Omega$; see, e.g.,  \cite[p. 92]{Ra}. Hence  $\{p_n(G^\star,z)/\gamma_n(G^\star)\}_{n=0}^\infty$
constitutes a sequence of asymptotically extremal monic polynomials on $\overline{G}$ and the
result follows from \cite[Theorem 2.3]{MhSa91}; see also \cite[Theorem III.4.7]{ST} and
\cite[Theorem~1.1]{SaSt2015}
\qed

\subsection{Proof of Proposition~\ref{thm:Ex3}}
Let $L^2_a(\mathbb{D})$ denote the Bergman space of functions analytic and square integrable
in $\mathbb{D}$.
We shall need to work with the Hilbert space $L^{2\#}_a(\mathbb{D})$ 
consisting of functions in $L^{2}_a(\mathbb{D})$,
but with inner product
\begin{equation}
\langle f,g\rangle_{L^2(G^\star)}:=\int_{G^\star}f(w) \overline{g(w)} dA(w)
\end{equation}
and corresponding norm $\|\cdot\|_{L^2(G^\star)}$.
Since the two norms $\|\cdot\|_{L^2({\mathbb{D})}}$ and ${\|\cdot\|_{L^2(G^\star)}}$ are equivalent; see, e.g.,
\cite[p.~2447]{SSST}, it follows that the two spaces $L^2_a(\mathbb{D})$ and $L^{2\#}_a(\mathbb{D})$ contain
exactly the same set of functions. Furthermore,
the polynomial sequence $\{p_n(G^\star,\cdot)\}_{n=0}^\infty$ forms a complete orthonormal
system in $L^{2\#}_a(\mathbb{D})$ and there exists a kernel function $K^\star(\cdot,\zeta)\in L^{2}_a(\mathbb{D})$ with 
the reproducing property, for all $\zeta\in\mathbb{D}$,
\begin{equation}\label{eq:Krep}
f(\zeta)=\langle f, K^\star(\cdot,\zeta)\rangle_{L^2(G^\star)},\quad\forall f\in L^2_a(\mathbb{D}).
\end{equation}
Consequently,  $K^\star(z,\zeta)$ is given in terms of the orthonormal polynomials in $L^{2\#}_a(\mathbb{D})$ by
\begin{equation}\label{eq:Ksharp}
K^\star(z,\zeta)=\sum_{k=0}^\infty \overline{p_k(G^\star,\zeta)}p_k(G^\star,z),
\end{equation}
where the series on the right hand
side converges  uniformly on compact subsets of $\mathbb{D}\times \mathbb{D}$.

The analytic continuation properties of the kernel $K^\star(\cdot,\zeta)$, $\zeta\in\mathbb{D}$, across the unit circle,
play an essential role in our analysis. To this end, for an analytic function $f$ in $\mathbb{D}$ we define
\begin{equation}\label{eq:rhodef}
\rho(f):=\sup\{R: f(z) \mbox{ has an analytic continuation to } |z|<R\}.
\end{equation}
Note that $1\le R\le\infty$.

With the notation above, it follows from \eqref{eq:Ksharp}, the equivalence of the two norms $\|\cdot\|_{L^2({\mathbb{D})}}$ and ${\|\cdot\|_{L^2(G^\star)}}$ 
 and Walsh maximal convergence theory (see \cite[pp.~130--131]{Wa}) that
\begin{equation}\label{eq:rhoK}
\limsup_{n\to\infty}|p_n(G^\star,\zeta)|^{1/n}=\frac{1}{\rho(K^\star(\cdot,\zeta))},\quad \zeta\in\mathbb{D}.
\end{equation}
(Note that \eqref{eq:rhoK} is an analogue of the Cauchy-Hadamard formula.)
This implies that  the analytic continuation properties of the kernel $K^\star(\cdot,\zeta)$ determine the  \lq\lq$\limsup$\rq\rq  
behaviour of the orthonormal polynomials
$p_n(G^\star,\zeta)$, for $\zeta\in\mathbb{D}$.

We denote by  $\phi$ a conformal mapping 
\begin{equation}\label{eq:mob}
w=\phi(z)=(z-z_1)/(1-z/z_2)
\end{equation}
of $G^\star$ onto the annulus $A_r:=\{w:r<|w|<1\}$,
where $r$ is called the \textit{conformal module} of $G^\star$. 
Then, by \cite[Lemma~3.2]{SSST},
\begin{equation}\label{Krep-in-phi}
K^\star(z,\zeta)=\overline{\phi^\prime(\zeta)}{\phi^\prime(z)}\sum_{n=0}^\infty\frac{r^{2n}}{\pi[1-r^{2n}\overline{\phi(\zeta)}{\phi(z)}]^2},
\end{equation}
which is valid for any $\zeta\in\mathbb{D}$. 

Therefore, by using the form of $\phi$, we conclude that $K^\star(z,\zeta)$ is a meromorphic
function with poles at $z=z_2$ and at the points $z=\phi^{-1}\left(1/{r^{2n}\overline{\phi(\zeta)}}\right)$, $n=0,1,2,\ldots$.
Since by symmetry we have $\phi^{-1}\left(1/{\overline{\phi(\zeta)}}\right)=1/{\overline{\zeta}}$, $\zeta\in\mathbb{D}$,
it follows easily that
\begin{equation}\label{eq:rho-zeta}
\rho(K^{\star}(\cdot,\zeta))=\left\{
\begin{array}{cl}
\frac{1}{|\zeta|}, &\mathrm{if}\ |\zeta|\ge |z_1|, \\
 |z_2|,            &\mathrm{if}\ |\zeta|<|z_1|.
\end{array}
\right.
\end{equation}
The latter equation, in conjunction with \eqref{eq:zeta12} and \eqref{eq:rhoK}, leads to 

\begin{equation}\label{eq:pn-limsup-zeta}
\limsup_{n\to\infty}|p_n(G^\star,\zeta)|^{1/n}=\left\{
\begin{array}{cl}
|\zeta|, &\mathrm{if}\ |\zeta|\ge |z_1| ,\\
 |z_1|,            &\mathrm{if}\ |\zeta|<|z_1|.
\end{array}
\right.
\end{equation}

Consider now the sequence of monic polynomials
$$
q_n(z):=\frac{p_n(G^\star,z)}{\gamma_n(G^\star)}=z^n+\cdots,\quad n=0,1,\ldots,
$$
and note that, in view of \eqref{eq:SaTo}, we have
$$
\lim_{n\to\infty}\gamma_n(G^\star)^{1/n}=1,
$$
where we made use of the fact that $\capD=1$. 

Now let $E$ denote the closed disk $|z|\le|z_1|$ and set $g(z)=z$. Then, it follows from \eqref{eq:pn-limsup-zeta} that
\begin{equation}\label{eq:dE}
\limsup_{n\to\infty}|q_n(z)|^{1/n}= |g(z)|,\quad z \in\partial E.
\end{equation}
Furthermore, for any $\zeta$ in the interior of  $E$,
there exists a subsequence $\mathcal{N}_{\zeta}$ of $\mathbb{N}$ such that
\begin{equation}\label{eq:qn-lim-zeta}
\lim_{n\to\infty}|q_n(\zeta)|^{1/n}=|z_1|,\quad  n\in\mathcal{N}_{\zeta},
\end{equation}

Finally, we note that the equilibrium measure of $E$ coincides with $\mu_{|z_1|}$, the normalized arclength measure of the circle $|z|=|z_1|$.
Since the \textit{logarithmic potential} $U^{\mu_{|z_1|}}(z)$ of $\mu_{|z_1|}$ satisfies (see, e.g., \cite[p.~45]{ST}):
\begin{equation*}\label{eq:arcmeas}
U^{\mu_{|z_1|}}(z)=\left\{
\begin{array}{cl}
\log\frac{1}{|z_1|}, &\mathrm{if}\ z\in\mathbb{C}\setminus E,\\
\log\frac{1}{|z|},     &  \mathrm{if}\ z\in E,
\end{array}
\right.
\end{equation*}
the result \eqref{eq:cic} then follows from \eqref{eq:dE}--\eqref{eq:qn-lim-zeta} by an application of Lemma~4.3 in \cite{Mina2013}.

\qed

\bigskip
{\bf{Acknowledgements.}} The authors are grateful to the two referees for
their helpful suggestions, especially for a shortened proof of Proposition~\ref{lem:diagkernelout}.

\bibliographystyle{amsplain}

\providecommand{\bysame}{\leavevmode\hbox to3em{\hrulefill}\thinspace}
\providecommand{\MR}{\relax\ifhmode\unskip\space\fi MR }
\providecommand{\MRhref}[2]{%
  \href{http://www.ams.org/mathscinet-getitem?mr=#1}{#2}
}
\providecommand{\href}[2]{#2}

\end{document}